\documentclass{amsart}
\usepackage[latin1]{inputenc}  
\usepackage[T1]{fontenc}       
\usepackage{amsmath,amsthm,amssymb,mathrsfs}
\usepackage{times}
\usepackage{enumerate} 
\usepackage{amsfonts}
\usepackage{calligra}
\usepackage{appendix}
\usepackage{graphicx}
\usepackage{subfig}
\usepackage[english]{babel}
\usepackage{mathtools}
\usepackage{tensor}
\usepackage{graphicx}
\usepackage{subfig}
\usepackage{xcolor}
\usepackage{stmaryrd}

\def\R{\mathbb{R}}
\def\N{\mathbb{N}}

\def\T{\mathbb{T}}

\newcommand{\norm}[1]{\left\lVert#1\right\rVert}
\newcommand{\abs}[1]{\left\lvert#1\right\rvert}

\def \tp{\textup}
\def\m{\mathbf{m}}
\def\v{\mathbf{v}}
\def\w{\mathbf{w}}
\def\e{\mathbf{e}}
\def\n{\nu}
\def\0{\mathbf{0}}
\def\d{\,\textup{d}}

\newcommand{\defeq}{\mathrel{:\mkern-0.25mu=}}
\newcommand{\eqdef}{\mathrel{=\mkern-0.25mu:}}

\newtheorem{thm}{Theorem}[section]
\newtheorem{cor}[thm]{Corollary}
\newtheorem{prop}[thm]{Proposition}

\theoremstyle{definition}

\newtheorem{rem}[thm]{Remark}

\numberwithin{equation}{section}

\begin{document}

\title{The lamination convex hull of stationary IPM}

\author{Lauri Hitruhin}
\address{Aalto University, Department of Mathematics and Systems Analysis, P.O. Box 11100, FI-00076 Aalto, Finland}
\email{lauri.hitruhin@aalto.fi}
\thanks{L.H. was supported by ICMAT Severo Ochoa project SEV-2015-0554 grant MTM2017-85934-C3-2-P and the
	ERC grant 307179-GFTIPFD and the ERC grant 834728 Quamap and by a grant from The Emil Aaltonen Foundation. S.L. was supported by the ERC grant 307179-GFTIPFD and by the AtMath Collaboration at the University of Helsinki.}

\author{Sauli Lindberg}
\address{University of Helsinki, Department of Mathematics and Statistics, P.O. BOX 68, 00014 Helsingin yliopisto, Finland}
\email{sauli.lindberg@helsinki.fi}
\thanks{}

\date{}

\begin{abstract}
We compute the lamination convex hull of the stationary IPM equations. We also show in bounded domains that for subsolutions of stationary IPM taking values in the lamination convex hull, velocity vanishes identically and density depends only on height. We relate the results to the infinite time limit of non-stationary IPM.
\end{abstract}

\maketitle
\section{Introduction}
We consider the flow of two immiscible incompressible fluids with equal viscosities and different densities in a porous medium. This can be modelled by the incompressible porous media equations (IPM) which consist of conservation of mass, incompressibility and Darcy's law:
\begin{align}
& \partial_t \rho + \nabla \cdot (\rho \v) = 0, \label{True IPM 1} \\
& \nabla \cdot \v = 0, \label{True IPM 2} \\
& \frac{\mu}{\kappa} \v = - \nabla p - \rho \mathbf{g} \label{True IPM 3},
\end{align}
where $\rho(x,t) \in \R$ is the fluid density, $\v(x,t) \in \R^2$ is the fluid velocity and $\mathbf{g} = (0,g)$ is gravity ~\cite{CG}. In the case of a smooth (simply connected) domain $\Omega \subset \R^2$, we assume the impermeability condition $\v \cdot \n = 0$ on $\partial \Omega$. Without loss of generality, we set $\mu/\kappa = g = 1$.

C\'{o}rdoba, Faraco and Gancedo proved the non-uniqueness of spatially periodic weak solutions $\v = (v_1,v_2) \in L^\infty(0,T;L^2)$ and $\rho \in L^\infty(0,T;L^\infty)$ of IPM in ~\cite{CFG}. The proof employs the method of convex integration which was first adapted to hydrodynamics by De Lellis and Sz\'{e}kelyhidi in their ground-breaking paper ~\cite{DLS}. The construction of ~\cite{CFG}, built on degenerate T4 configurations, provides a robust method of constructing bounded weak solutions in inviscid fluid dynamics without determining the exact $\Lambda$-convex hull; for an application to more general active scalar equations with an even multiplier see ~\cite{Shvydkoy}. The existence and non-uniqueness of spatially periodic $C^{1/9-\epsilon}$ solutions of IPM for smooth initial datas was shown by Isett and Vicol in ~\cite{IV}.

In ~\cite{Sze}, Sz\'{e}kelyhidi computed the $\Lambda$-convex hull and showed it to be the exact relaxation of IPM equations. He also used it to construct infinitely many admissible weak solutions to the unstable Muskat problem in $\Omega = (-1,1)^2$ with a flat interface as initial data. Sz\'{e}kelyhidi also computed a differently normalised hull that leads to solutions with a bounded velocity. Admissible mixing solutions to the unstable Muskat problem with a non-flat $H^5$-regular interface were constructed by Castro, C\'{o}rdoba and Faraco in ~\cite{CCF}. For further developments see ~\cite{ACF,CFM,FSz,Mengual,NSz}; here, also, the construction of admissible weak solutions relies on the exact hull and the construction of an admissible subsolution.

In stationary IPM, in contrast, if $\v \in L^2(\Omega,\R^2)$ with $\v \cdot \n|_{\partial \Omega} = 0$ and $\rho \in L^\infty(\Omega)$ form a weak solution, then $\v \equiv 0$ and $\partial_1 \rho \equiv 0$; a proof of this simple fact by Elgindi appears in ~\cite{Elgindi}. As a main result, Elgindi showed on $\R^2$ and $\T^2$ that whenever solutions of non-stationary IPM have initial datas near certain stationary solutions, they must converge to the stationary solution in $H^3$ when $t \to \infty$. The global well-posedness of non-stationary IPM is open, but Elgindi showed it around said stationary solutions. In \cite{CCL}, Castro, C\'{o}rdoba and Lear proved structurally similar results for the \emph{confined IPM} case $\Omega = \T^1 \times (-1,1)$, overcoming new difficulties to do with the boundary.

Nevertheless, in ~\cite{CLV}, Constantin, La and Vicol constructed solutions of stationary IPM that are smooth and vanish outside a strip that has finite width in the direction $x_1 = k x_2$, $k \in \R$. The result is one example of their construction which uses Grad-Shafranov-like equations to obtain smooth, localised solutions in hydrodynamics, motivated by Gavrilov's construction of smooth, compactly supported solutions of stationary Euler equations in ~\cite{Gavrilov}. The solutions of stationary IPM in ~\cite{CLV} are functions of the variable $z = x_1 - k x_2$ and, as such, they are periodic (even constant) in the axial direction of the strip. If the direction of finite width of the strip is $(0,1)$, i.e., in the case $\Omega = \mathbb{T}^1 \times (-1,1)$, an easy adaptation of Elgindi's proof (see \textsection \ref{Non-existence of non-trivial solutions in bounded domains}) rules out such a construction. This dichotomy highlights the role of the direction of gravity in IPM and is discussed briefly in Remark \ref{Dichotomy remark}.

\vspace{0.2cm}
The problem we address is the determination of the \emph{relaxation} of stationary IPM. One of our aims is to shed light on the following question: how are the differences between stationary and non-stationary IPM as well as the somewhat surprising combination of the results of ~\cite{CLV} and ~\cite{Elgindi} reflected in the relaxation? We also wish to use information on the relaxation to better understand the infinite time limit of non-stationary IPM.

We set the stage by briefly describing convex integration in the Tartar framework; the relevant definitions are recalled in \textsection \ref{Relevant notions}. One first decouples a system of non-linear constant-coefficient PDE's into a system of first-order linear PDE's $\mathscr{L}(z) = \0$ and the pointwise constraint that $z(x)$ takes values in a constitutive set $K$. In the case of stationary IPM, $z = (\rho,\v,\m)$, the set of linear equations $\mathscr{L}(z) = \0$ is
\begin{align}
& \nabla \cdot \m = 0, \label{IPM 1} \\
& \nabla \cdot \v = 0, \label{IPM 2} \\
& \nabla^\perp \cdot (\v + (0,\rho)) = 0 \label{IPM 3}
\end{align}
and the constitutive set is
\begin{equation} \label{IPM 4}
K = \{(\rho,\v,\m) \in \R \times \R^2 \times \R^2 \colon \abs{\rho} = 1, \; \m = \rho \v\},
\end{equation}
where the constraint $\rho \in \{-1,1\}$ codifies the densities of the two immiscible fluids. Returning to the general Tartar framework, given initial/boundary datas, one attempts to construct a strict \emph{subsolution}, that is, $z_0$ satisfying $\mathscr{L}(z_0) = \0$ and taking values in a suitable subset $\mathscr{U}$ of the $\Lambda$-convex hull $K^\Lambda$; usually, $\mathscr{U} = \tp{int} (K^\Lambda)$. One then forms perturbations $z_j$ of $z_0$ by adding localised plane waves where the admissible directions of oscillation are dictated by the \emph{wave cone} $\Lambda$ and $z_j$ take values in $\mathscr{U}$. By a limiting argument, one intends to find infinitely many subsolutions with the prescribed initial/boundary conditions and with values in $K$, which would then yield non-uniqueness of the original system of PDE's for the given boundary/initial datas (see ~\cite{DLS12}).

The \emph{relaxation} of $K$ can be given slightly different meanings, but it is defined here as the smallest set $\tilde{K} \supset K$ that is stable under weak convergence for solutions of \eqref{IPM 1}--\eqref{IPM 3}, essentially following Tartar ~\cite{Tartar}. As such, it models macroscopic averages of solutions of stationary IPM. By a result of Tartar, $\tilde{K}$ contains the lamination convex hull $K^{lc,\Lambda}$ ~\cite[Theorem 8]{Tartar}. We compute the lamination convex hull of stationary IPM in Theorem \ref{Main theorem}, and we believe that as in non-stationary IPM, the lamination and $\Lambda$-convex hulls and the relaxation coincide.

\vspace{0.2cm}
As emphasised  in ~\cite{DLS12,Sze}, precise information on the hull is crucial in identifying the boundary/initial datas for which one can run convex integration. As an example of this we mention that the hull of compressible Euler is notoriously difficult to compute and that to the authors' knowledge, due to insufficient information on the hull, lack of uniqueness has so far only been shown for a set of datas where one is able to reduce to an incompressible system; see ~\cite{FKM, Markfelder}. However, in ~\cite{Markfelder}, Markfelder computed the $\Lambda$-convex hull of a suitably normalised constraint set $K$.

Furthermore, the physical relevance of subsolutions was already emphasised in ~\cite{Sze} in the case of the Muskat problem. The unstable Muskat problem with a flat interface is ill-posed, but in a pioneering work ~\cite{Otto}, Otto had used mass transport techniques to construct macroscopically averaged relaxed solutions that arise as an entropy solution of a scalar conservation law. At a certain asymptotic limit ~\cite[p. 505]{Sze}, Sz\'{e}kelyhidi's subsolutions converge to Otto's relaxed solution. A subsolution can be viewed as a kind of coarse-grained average; this interpretation is explored in detail e.g. in ~\cite{CFM,DLS12,Sze}.

(Topological) smallness of the hull seems to reflect uniqueness of bounded solutions under trivial initial/boundary datas and (in the case of evolutionary models) existence of robust conserved quantities. As an example, in IPM and other active scalar equations with an even Fourier multiplier, $K^\Lambda$ has a non-empty interior ~\cite{Knott,Shvydkoy} and there exist non-trivial bounded (even H\"{o}lder continuous) solutions with compact support in time ~\cite{Shvydkoy,IV}. SQG, in contrast, has an odd multiplier and a trivial hull $K^\Lambda = K$ (defining $\Lambda$ as in ~\cite{Knott,Shvydkoy}), and the Hamiltonian is conserved by $L^3$ solutions, ruling out bounded solutions with compact support in time ~\cite{IV}.

Quadratic $\Lambda$-affine functions are a simple and powerful tool in determining the size of $K^\Lambda$. To illustrate this, while 2D and 3D ideal MHD look superficially similar to Euler equations, both possess a non-trivial quadratic $\Lambda$-affine function which vanishes in $K$, making $\operatorname{int}(K^\Lambda)$ empty. As a direct reflection of this, bounded solutions conserve the mean-square magnetic potential in 2D and the magnetic helicity in 3D. This rules out solutions with a non-trivial, compactly supported magnetic field in 2D but, perhaps surprisingly, not in 3D ~\cite{FLSz}. By Tartar's Theorem (see ~\cite[Theorem 11]{Tartar}), quadratic $\Lambda$-affine functions are weakly continuous, and as such, they also aid the understanding of various asymptotic regimes such as weak limits of (sub)solutions or the inviscid limit; see also ~\cite[p. 58]{CFM}.

\vspace{0.2cm}
In non-stationary IPM, \eqref{IPM 1} is replaced by $\partial_t \rho + \nabla \cdot \m = 0$ and the $\Lambda$-convex hull consists of triples $(\rho,\v,\m)$ such that $\abs{\rho} \le 1$ and $\abs{\m - \rho \v + \left( 0, (1-\rho^2)/2 \right)} \le (1-\rho^2)/2$ ~\cite{Sze}. In particular, the $\Lambda$-convex hull has a non-empty interior.

In stationary IPM, however, $G(\rho,\v,\m) \defeq \m \cdot \v^\perp$ vanishes in $K \cap \Lambda$, enforcing $\operatorname{int}(K^\Lambda) = \emptyset$. Other quadratic $\Lambda$-affine functions of stationary IPM include $\abs{\v}^2 + \rho v_2$ and $\m \cdot (\v + (0,\rho))$--in fact, these three functions determine $\Lambda$ (see Proposition \ref{Wave cone proposition}). If $(\rho,\v,\m) \in K^\Lambda$ with $\v \neq \0$, then $\m \cdot \v^\perp = 0$ yields $\m = k \v$ for some $k \in \R$. The main challenge in the computation of $K^{lc,\Lambda}$ is the determination of the exact range of the constant of proportionality $k$ in $\m = k \v$.

\begin{thm} \label{Main theorem}
$K^{lc,\Lambda} = \cup_{j=1}^4 X_j$, where
\begin{align*}
X_1 &\defeq \left\{ \left( \rho, \0, \frac{1-\rho^2}{2} [\e-(0,1)] \right) \colon \abs{\rho} \le 1, \, \abs{\e} \le 1 \right\}, \\
X_2 &\defeq \left\{ (\rho, \v, k\v) \colon \abs{\rho} \le 1, \, \v \neq \0, \, 1 \le k \le \rho - \frac{(1-\rho^2) v_2}{\abs{\v}^2} \right\}, \\
X_3 &\defeq \left\{ (\rho, \v, k\v) \colon \abs{\rho} < 1, \, \v \neq \0, \, -1 < k = \rho - \frac{(1-\rho^2) v_2}{\abs{\v}^2}  < 1 \right\}, \\
X_4 &\defeq \left\{ (\rho, \v, k\v) \colon \abs{\rho} \le 1, \, \v \neq \0, \, \rho - \frac{(1-\rho^2) v_2}{\abs{\v}^2} \le k \le -1 \right\}.
\end{align*}
\end{thm}

We interpret $K^{lc,\Lambda}$ geometrically. The projections of $X_2$ and $X_4$ into $\R \times \R^2$ are cones where $\rho \in [-1,1]$ and
\begin{align}
& \abs{\v + \left( 0, \frac{1+\rho}{2} \right)} \le \frac{1+\rho}{2} \quad \Longleftrightarrow \quad \abs{\v}^2 + (\rho+1) v_2 \le 0, \quad (X_2) \label{Cone 1} \\
& \abs{\v - \left( 0, \frac{1-\rho}{2} \right)} \le \frac{1-\rho}{2} \quad \Longleftrightarrow \quad \abs{\v}^2 + (\rho-1) v_2 \le 0. \quad (X_4) \label{Cone 2}
\end{align}
The \emph{power balance} $\abs{\v}^2 + \rho v_2$ can be interpreted as the balance between the density of energy per unit time consumed by friction and the density of work per unit time done by gravity ~\cite{CFM}. Thus Theorem \ref{Main theorem} shows that $\{(\rho,\v,\m) \in K^{lc,\Lambda} \colon \v \neq \0\}$ divides into two subsets: the \emph{flexible region} where $\abs{\v}^2 + \rho v_2 \le \abs{v_2}$ (the cones) and the parameter $k$ in $\m = k \v$ lies on a non-degenerate interval, and the \emph{rigid region}  where the power balance dominates the vertical speed $\abs{v_2}$ and $k$ is uniquely determined. This rigid region is just the projection of $X_3$ into $\R \times \R^2$. When $(\rho,\0,\m) \in K^{lc,\Lambda}$, and thus $z=(\rho,\0,\m) \in X_1 $, the component $\m$ has the same range of values as in non-stationary IPM. Furthermore, the projection of $X_1$ into $\R \times \R^2$ is the line segment that is formed as an intersection of the cones resulting from $X_2$ and $X_4$.

The main technical difficulties of the proof involve the smallness of the set $X_3$. Note that for any suitable pair $(\rho, \v)$ there exists exactly one $\m\in \R^2$ such that $(\rho,\v,\m) \in X_3 $. This makes it very challenging to construct $\Lambda$-convex functions that would show for these $(\rho,\v)$ that the lamination convex and $\Lambda$-convex hull coincide. We nevertheless manage to show coincidence for all other points; see \eqref{Easier part of hull}. The difficulties are also present in Propositions \ref{Lemma outside the cones}--\ref{Lemma 3 outside the cones}, most notably when showing the lamination convexity of $X_3$; this is the technically most difficult part of the paper. 

\vspace{0.2cm}
As another main result, we show that if a subsolution of stationary IPM takes values in $K^{lc,\Lambda}$, it has a vanishing velocity. Recall that $L^2_\sigma(\Omega,\R^2) \defeq \{\w \in L^2(\Omega,\R^2) \colon \nabla \cdot \w = 0, \, \w \cdot \n|_{\partial \Omega} = 0\}$.

\begin{thm} \label{Second main theorem}
Suppose $\Omega \subset \R^2$ is bounded, strongly Lipschitz and simply connected. Suppose $\rho \in L^\infty(\Omega)$ and $\v,\m \in L^2_\sigma(\Omega,\R^2)$ satisfy \eqref{IPM 1}--\eqref{IPM 3}. Suppose $(\rho,\v,\m)(x) \in K^{lc,\Lambda}$ a.e. $x \in \Omega$. Then $\v = \0$ and $\partial_1 \rho = 0$.
\end{thm}

The conclusion of Theorem \ref{Second main theorem} also holds in $\T^1 \times (-1,1)$, extending the dichotomy on solutions vanishing outside a strip into subsolutions (see Remark \ref{Dichotomy remark}).

Motivated by Elgindi's computations in ~\cite{Elgindi} as well as Theorem \ref{Second main theorem}, we also discuss the infinite time limit of non-stationary IPM. We show in Proposition \ref{Third main theorem} that if $\rho \in L^\infty(0,\infty;L^\infty)$ and $\v,\m \in L^\infty(0,\infty;L^2_\sigma)$ form a subsolution in a bounded domain $\Omega$, then $\v \in L^2(0,\infty;L^2_\sigma)$; in particular, $\lim_{M \to \infty} \int_M^\infty \int_\Omega |\v(x,t)|^2 \d x \d t = 0$.

The structure of the paper is as follows. The relevant definitions are recalled in \textsection \ref{Relevant notions}, where we also compute the wave cone. The inclusion $K^{lc,\Lambda} \supset \cup_{j=1}^4 X_j$ is proved in \textsection \ref{Estimating the hull from below} whereas $K^{lc,\Lambda} \subset \cup_{j=1}^4 X_j$ is proved in \textsection \ref{Estimating the hull from above}. The proof of Theorem \ref{Second main theorem} is presented in \textsection \ref{Non-existence of non-trivial solutions in bounded domains}, and the limit $t \to \infty$ of non-stationary IPM is studied in \textsection \ref{Relation to the limit of non-stationary IPM}.

\section{Relevant notions} \label{Relevant notions}
We briefly recall some notions from the theory of differential inclusions; a thorough discussion of related topics can be found in ~\cite{Kirchheim}.

The \emph{wave cone} $\Lambda$  consists of directions $z = (\rho, \v, \m) \in \R \times \R^2 \times \R^2$ such that for some $\xi \in \R^2 \setminus \{\0\}$, plane waves of the form $x \mapsto h(x \cdot \xi) z \colon \R^2 \to \R \times \R^2 \times \R^2$ satisfy \eqref{IPM 1}--\eqref{IPM 3} for all $h \in C^\infty(\R)$. Denoting $(\xi_1,\xi_2)^\perp = (-\xi_2,\xi_1)$, the wave cone conditions are thus
\begin{align}
&  \m \cdot \xi = 0, \label{Wave cone 1} \\
& \v \cdot \xi = 0, \label{Wave cone 2} \\
& (\v + (0,\rho)) \cdot \xi^\perp = 0. \label{Wave cone 3}
\end{align}
An explicit form of $\Lambda$ is given in Proposition \ref{Wave cone proposition} and Corollary \ref{Wave cone corollary}.

\begin{prop} \label{Wave cone proposition}
The wave cone of stationary IPM is
\begin{align*}
\Lambda
&= \{(\rho,\v,\m) \colon \abs{\v + (0,\rho/2)} = \abs{\rho}/2, \quad \m \cdot \v^\perp = 0 \quad \text{and} \quad \m \cdot (\v + (0,\rho)) = 0 \}. 
\end{align*}
\end{prop}

\begin{proof}
First assume $(\rho,\v,\m) \in \Lambda$. The conditions \eqref{Wave cone 2}--\eqref{Wave cone 3} imply that $\v = k \xi^\perp$ and $\v+(0,\rho) = \ell \xi$ for some $k,\ell \in \R$. Thus $\abs{\v}^2 + \rho v_2 = k \ell \xi \cdot \xi^\perp = 0$, giving $\abs{\v + (0,\rho/2)} = \abs{\rho}/2$. If $\rho = 0$, then $\v = \0$ and so clearly $\m \cdot (\v + (0,\rho)) = \m \cdot \v^\perp = 0$. If $\rho \neq 0$, then \eqref{Wave cone 1}--\eqref{Wave cone 3} give $\m \cdot (\v + (0,\rho)) = \ell \m \cdot \xi = 0$ and $\m \cdot \v^\perp = - k \m \cdot \xi = 0$. 

Conversely, if $\abs{\v}^2 + \rho v_2 = \m \cdot (\v+(0,\rho)) = \m \cdot \v^\perp = 0$, then we get $(\rho,\v,\m) \in \Lambda$ by choosing $\xi = \v + (0,\rho)$ if $\v \neq (0,-\rho) \neq \0$, $\xi = \v^\perp$ if $\v=(0,-\rho) \neq \0$, $\xi =\m^\perp $ if $(0,\rho)=\0=\v$ and $\m\neq \0$, and finally $\xi=(1,1)$ if $\v=(0,\rho)=\m=\0$.
\end{proof}

\begin{cor} \label{Wave cone corollary}
The wave cone $\Lambda$ consists of vectors $z \in \R \times \R^2 \times \R^2$ of the following three forms:
\begin{align*}
& z = \left( \rho, \frac{\rho}{2} (\e-(0,1)), \ell (\e-(0,1)) \right), \quad \rho \neq 0, \, \e \in S^1 \setminus \{(0,1)\}, \, \ell \in \R, \\
& z = (\rho, \0, (m_1,0)), \quad \rho \neq 0, \, m_1 \in \R, \\
& z = (0,\0,\m), \quad \m \in \R^2.
\end{align*}
\end{cor}

\begin{rem} \label{Wave cone remark}
The first condition in Proposition \ref{Wave cone proposition} can be written as $\abs{\v}^2 + \rho v_2 = 0$.
\end{rem}

Given any compact set $C \subset \R \times \R^2 \times \R^2$, the \emph{laminates} $C^{k,\Lambda}$, $k \in \N_0$, of $C$ are defined as follows:
\begin{align*}
& C^{0,\Lambda} \defeq C, \\
& C^{k+1,\Lambda} \defeq \{(\lambda z_1 + (1-\lambda) z_2 \colon z_1,z_2 \in C^{k,\Lambda}, \, z_1-z_2 \in \Lambda, \, \lambda \in [0,1]\}. 
\end{align*}
The \emph{lamination convex hull} of $C$ is defined as
\[C^{lc,\Lambda} \defeq \cup_{k=0}^\infty C^{k,\Lambda}.\]
Recall also that a function $G \colon \R \times \R^2 \times \R^2 \to \R$ is said to be \emph{$\Lambda$-convex} if $t \mapsto G(z_0 + tz) \colon \R \to \R$ is convex for every $z_0 \in \R \times \R^2 \times \R^2$ and $z \in \Lambda$. The \emph{$\Lambda$-convex hull} $C^\Lambda$ consists of points $z \in \R \times \R^2 \times \R^2$ that cannot be separated from $C$ by a $\Lambda$-convex function. More precisely, $z \notin C^\Lambda$ if and only if there exists a $\Lambda$-convex function $G$ such that $G|_C \le 0$ but $G(z) > 0$. We have $C^\Lambda \supset C^{lc,\Lambda}$.

\begin{rem} \label{Remark on hulls}
Denote the wave cone of non-stationary IPM by $\Lambda_{\operatorname{ns}}$. The constitutive set $K$ is the same in stationary and non-stationary IPM but $\Lambda \subset \Lambda_{\operatorname{ns}}$, so that we immediately get $K^{lc,\Lambda} \subset K^{lc,\Lambda_{\operatorname{ns}}}$ and $K^\Lambda \subset K^{\Lambda_{\text{ns}}}$.
\end{rem}

If $\rho \in L^\infty(\Omega)$ and $\v, \m \in L^2_\sigma(\Omega,\R^2)$ satisfy \eqref{IPM 1}--\eqref{IPM 3} and $z(x) = (\rho,\v,\m)(x) \in K^\Lambda$ a.e. $x \in \Omega$, then $z$ is called a \emph{subsolution of stationary IPM}.

\section{Estimating the hull from below} \label{Estimating the hull from below}
We wish to first show that $K^{lc,\Lambda}$ contains the set $\cup_{j=1}^4 X_j$ described in Theorem \ref{Main theorem}. We begin by computing the first laminate.

\begin{prop} \label{Proposition on first laminate}
We have
\begin{align*}
K^{1,\Lambda}
&= \left\{ (\rho, \0, \m ) \colon \abs{\rho} \le 1, \quad \m = \frac{1-\rho^2}{2} (\e-(0,1)), \quad \abs{\e} = 1 \right\} \\
&\bigcup \left\{ (\rho, \v, \m) \colon \abs{\rho} \le 1, \quad \m = \left[ \rho - \frac{(1-\rho^2) v_2}{\abs{\v}^2} \right] \v \right\}.
\end{align*}
\end{prop}

\begin{proof}
A general convex combination of two elements of $K$ is either an element of $K$ or of the form
\begin{align}
\left( \rho, \v, \rho \v +(1-\rho^2) \w \right)
&= \frac{1+\rho}{2} \left( 1, \v + (1-\rho) \w, \v + (1-\rho) \w \right) \label{Lambda convex 1} \\
&+ \frac{1-\rho}{2} \left( -1, \v - (1+\rho) \w, -\v + (1+\rho) \w \right), \label{Lambda-convex 2}
\end{align}
where $\abs{\rho} < 1$ and $\v, \w \in \R^2$.

The linear combination in \eqref{Lambda convex 1}--\eqref{Lambda-convex 2} is $\Lambda$-convex if and only if $(2,2\w,2\v - 2 \rho \w) \in \Lambda$. By Proposition \ref{Wave cone proposition}, this occurs precisely when $\abs{\w+(0,1/2)} = 1/2$, $\w \cdot \v^\perp = 0$ and $\v \cdot (\w + (0,1)) = 0$.

If $\v = \0$, the wave cone conditions are equivalent to $\w = (\e-(0,1))/2$ with $\abs{\e} = 1$, whereas in the case $\v \neq \0$ they are equivalent to $\w = -(v_2/\abs{\v}^2) \v$, which completes the proof.
\end{proof}

By Corollary \ref{Wave cone corollary} and Proposition \ref{Proposition on first laminate}, $K^{lc,\Lambda} \supset X_1 \cup X_3$. The next two propositions, combined with Corollary \ref{Wave cone corollary}, show that $K^{lc,\Lambda} \supset X_2 \cup X_4$.

\begin{prop} \label{q for first cone}
Suppose $\abs{\rho} < 1$ and $\v \neq \0$ with
\begin{equation} \label{Cone condition clumsy}
\rho - \frac{(1-\rho^2) v_2}{\abs{\v}^2} \ge 1.
\end{equation}
Then
\[\left( \rho, \v, \v \right) \in K^{lc,\Lambda}.\]
\end{prop}

\begin{proof}
Suppose \eqref{Cone condition clumsy} holds. As a consequence, $v_2 < 0$. Let us write
\[\left( \rho, \v, \v \right) 
= \lambda \left( 1, \frac{\v}{\lambda}, \frac{\v}{\lambda} \right) + (1-\lambda) (\psi, \0, \0),\]
where $0 < \lambda < 1$ and
\begin{equation} \label{psi}
\psi = \frac{\rho-\lambda}{1-\lambda}.
\end{equation}
We need to choose $\lambda$ in such a way that $-1 \le \psi < \rho$ and $z_1-z_2 = ( 1-\psi, \v/\lambda, \v/\lambda) \in \Lambda$.

By Proposition \ref{Wave cone proposition} and Remark \ref{Wave cone remark}, $z_1-z_2 \in \Lambda$ is equivalent to
\[\frac{\v}{\lambda} \cdot \left[ \frac{\v}{\lambda} + (0,1-\psi) \right] = 0.\]
In conjunction with \eqref{psi}, this leads to the choices
\[\lambda = \frac{\abs{\v}^2}{\abs{\v}^2 - (1-\rho) v_2}, \qquad \psi = \frac{\abs{\v}^2 + \rho v_2}{v_2}.\]
Note that \eqref{Cone condition clumsy} holds if and only if $\abs{\v}^2 + (1+\rho) v_2 \le 0$ if and only if $\psi \ge -1$. Since $v_2 < 0$, we also have $0 < \lambda < 1$. Furthermore, $\psi = (\rho-\lambda)/(1-\lambda) < \rho$ since $\rho < 1$.
\end{proof}

\begin{prop} \label{q for second cone}
Suppose $\abs{\rho} \le 1$ and $\v \neq \0$ with
\begin{equation} \label{Cone condition clumsy 2}
\rho - \frac{(1-\rho^2) v_2}{\abs{\v}^2} \le - 1.
\end{equation}
Then
\[\left( \rho, \v, -\v \right) \in K^{lc,\Lambda}.\]
\end{prop}

\begin{proof}
The proof is entirely analogous to that of Proposition \ref{q for first cone}; we write $(\rho,\v,-\v) = \lambda (-1,\v/\lambda,-\v/\lambda) + (1-\lambda) (\psi,\0,\0)$ and set
\[\lambda = \frac{\abs{\v}^2}{\abs{\v}^2 + (1+\rho) v_2}, \qquad \psi = \frac{\abs{\v}^2 + \rho v_2}{v_2}.\]
Now \eqref{Cone condition clumsy 2} is equivalent to $\abs{\v}^2 - (1-\rho) v_2 \le 0$, which in turn is equivalent to $\psi \le 1$. In addition, \eqref{Cone condition clumsy 2} implies $v_2 > 0$, which in turn gives $0 < \lambda < 1$.
\end{proof}

\section{Estimating the hull from above} \label{Estimating the hull from above}
We now intend to show that $K^{lc,\Lambda} \subset \cup_{j=1}^4 X_j$. The steps of the proof are as follows: 

$\bullet$ when $\v = \0$, Corollary \ref{Proposition on m for u = (0,rho)} shows that if $z = (\rho,\0,\m) \in K^\Lambda$, then $z \in X_1$.

$\bullet$ when $(\rho,\v,\m) \in K^\Lambda$ and $\v \neq \0$, Corollary \ref{Corollary on k} shows that $\m = k\v$ for some $k \in \R$. 

$\bullet$ When $z = (\rho,\v,k\v) \in K^\Lambda$ and $|\rho - (1-\rho^2) v_2/\abs{\v}^2| \ge 1$, Corollary \ref{Corollary on Lambda-convex function for first cone} yields $z \in X_2 \cup X_4$.

$\bullet$ When $z = (\rho,\v,k\v) \in K^{lc,\Lambda}$ and $|\rho - (1-\rho^2) v_2/\abs{\v}^2| < 1$, Propositions \ref{Lemma outside the cones}--\ref{Lemma 3 outside the cones} imply that $z \in X_3$. This is the only result that we are not able to show for $K^\Lambda$ but only for $K^{lc,\Lambda}$.

\vspace{0.2cm}
We begin by recalling a proposition from ~\cite{Sze} which also applies to stationary IPM in view of Remark \ref{Remark on hulls}:

\begin{prop} \label{Laszlo's proposition}
The function
\[G_1(\rho,\v,\m) \defeq \abs{\m - \rho \v + \left( 0,\frac{1-\rho^2}{2} \right)} - \frac{1-\rho^2}{2}\]
is $\Lambda$-convex and vanishes in $K$. Consequently,
\[K^\Lambda \subset \left\{ (\rho,\v,\m) \colon \abs{\rho} \le 1, \abs{\m - \rho \v + \left( 0, \frac{1-\rho^2}{2} \right)} \le \frac{1-\rho^2}{2} \right\}.\]
\end{prop}

Propositions \ref{Proposition on first laminate} and \ref{Laszlo's proposition} have the following consequence.

\begin{cor} \label{Proposition on m for u = (0,rho)}
Let $\abs{\rho} \le 1$. Then
\[(\rho, \0, \m) \in K^\Lambda \quad \Longleftrightarrow \quad \m = \frac{1-\rho^2}{2} (\e-(0,1)), \; \abs{\e} \le 1.\]
\end{cor}

We then consider the case $\v \neq \0$. The following result follows immediately from Proposition \ref{Wave cone proposition}.

\begin{prop} \label{Lambda-convex function proposition}
The function
\[G_2(\rho,\v,\m) \defeq \m \cdot \v^\perp\]
is $\Lambda$-affine and vanishes in $K$.
\end{prop}

\begin{cor} \label{Corollary on k}
If $(\rho, \v, \m) \in K^\Lambda$ with $\v \neq \0$, then $\m = k \v$ for some $k \in \R$.
\end{cor}

In view of Proposition \ref{Proposition on m for u = (0,rho)} and Corollary \ref{Corollary on k}, the hull $K^{lc,\Lambda}$ is determined by finding the exact range of the parameter $k = k(\rho,\v)$ in $\m = k \v$. Proposition \ref{Laszlo's proposition} implies that $k$ lies between $\rho$ and $\rho - (1-\rho^2) v_2/\abs{\v}^2$, giving the optimal range in the case of non-stationary IPM. However, in the case of \textit{stationary} IPM, the range of $k(\rho,\v)$ is smaller, as stated in Theorem \ref{Main theorem}.

We divide the set of points $(\rho,\v) \in \R \times (\R^2 \setminus \{\0\})$ into the cones described by \eqref{Cone 1}--\eqref{Cone 2} and the complement of their union. We first address the points of the two cones.

\begin{prop} \label{Proposition on Lambda-convex function for first cone}
The functions defined by
\begin{align*}
& G_3(\rho,\v,\m) \defeq -\left[ \v - \m \right] \cdot \left[ \v + (0,1+\rho) \right] + \frac{\abs{\v-\m}^2}{2}, \\
& G_4(\rho,\v,\m) \defeq - \left[ \v + \m \right] \cdot \left[ \v - (0,1-\rho) \right] + \frac{\abs{\v+\m}^2}{2},
\end{align*}
are $\Lambda$-convex and satisfy $G_3|_K = G_4|_K = 0$.
\end{prop}

\begin{proof}
We prove the claims for $G_3$; the proofs for $G_4$ are analogous. Let us fix $z_0 = (\rho_0,\v_0,\m_0) \in \R \times \R^2 \times \R^2$, $z = (\rho,\v,\m) \in \Lambda$ and $t \in \R$. Then
\begin{align*}
G_3(z_0 + t z)
&= -[\v_0 - \m_0 + t (\v-\m)] \cdot [\v_0 + (0, 1+\rho_0) + t (\v + (0,\rho))] \\
&+ \frac{\abs{\v_0 - \m_0 + t (\v-\m)}^2}{2} \\
&= G_3(z_0) + C_{z_0,z} t + \frac{\abs{\v-\m}^2}{2} t^2
\end{align*}
in view of Proposition \ref{Wave cone proposition}. Furthermore, $G_3(1,\v,\v) = 0$ and $G_3(-1,\v,-\v) = 0$ for all $\v \in \R^2$ so that $G_3|_K = 0$.
\end{proof}

\begin{cor} \label{Corollary on Lambda-convex function for first cone}
Suppose $\abs{\rho} < 1$, $\v \neq \0$ and $(\rho,\v,\m) \in K^\Lambda$. If $|\rho - (1-\rho^2)v_2/\abs{\v}^2| \ge 1$, then $z \in X_2 \cup X_4$.
\end{cor}

\begin{proof}
Assume $\rho - (1-\rho^2) v_2/\abs{\v}^2 \ge 1$; the proof of the case $\rho - (1-\rho^2) v_2/\abs{\v}^2 \le -1$ is analogous. By Corollary \ref{Corollary on k}, $\m = k \v$ for some $k \in \R$. Our aim is to show that $(\rho,\v,k\v) \in X_2$, i.e., $1 \le k \le \rho - (1-\rho^2) v_2/\abs{\v}^2$.

The inequality $k \le \rho - (1-\rho^2) v_2/\abs{\v}^2$ follows from Proposition \ref{Laszlo's proposition}. For the claim $k \ge 1$ note that $\rho - (1-\rho^2)v_2/\abs{\v}^2 \ge 1$ can be written as $(1-\rho) \abs{\v}^2 + (1-\rho^2) v_2 \le 0$. We compute
\begin{align*}
0
&\ge (1-\rho) G_3(\rho,\v,\m) \\
&= (k-1) \v \cdot \left( (1-\rho) \v + (0,1-\rho^2) + \frac{(1-\rho) (k-1) \v}{2} \right) \\
&= (k-1) \left( (1-\rho) \abs{\v}^2 + (1-\rho^2) v_2 + \frac{(1-\rho) (k-1) \abs{\v}^2}{2} \right),
\end{align*}
which implies the claim.
\end{proof}

Corollaries \ref{Proposition on m for u = (0,rho)} and \ref{Corollary on Lambda-convex function for first cone} show that
\begin{equation} \label{Easier part of hull}
K^\Lambda \setminus \left\{ (\rho,\v,\m) \colon \v \neq \0, \, -1 < \rho - \frac{(1-\rho^2) v_2}{\abs{\v}^2} < 1 \right\} = X_1 \cup X_2 \cup X_4.
\end{equation}
In other words, we have computed the exact range of the $\m$ component in all cases except $\v \neq \0$, $\rho - (1-\rho^2) v_2/\abs{\v}^2 \in (-1,1)$. We finish the proof of Theorem \ref{Main theorem} by showing that
\begin{equation} \label{Harder part of hull}
K^{lc,\Lambda} \cap \left\{ (\rho,\v,\m) \colon \v \neq \0, \, -1 < \rho - \frac{(1-\rho^2) v_2}{\abs{\v}^2} < 1 \right\} = X_3;
\end{equation}
combining \eqref{Easier part of hull} and \eqref{Harder part of hull} yields $K^{lc,\Lambda} = \cup_{j=1}^4 X_j$.

The proof of \eqref{Harder part of hull} consists of two parts. First, Proposition \ref{Lemma outside the cones} says that $X_3^{1,\Lambda} = X_3$. Then, if $z = (\rho,\v,k\v) \in (\cup_{j=1}^4 X_j)^{1,\Lambda}$, where $\v \neq \0$ and $-1 < \rho - (1-\rho^2) v_2/\abs{\v}^2 < 1$, we write $z$ as a $\Lambda$-convex combination of $z_1 \in X_i$ and $z_2 \in X_j$, where $i,j \in \{1,2,3,4\}$. We show in Propositions \ref{Lemma 2 outside the cones}--\ref{Lemma 3 outside the cones} that we cannot have $i \neq j$. Now, since each $X_i$ is lamination convex, we get $i = j = 3$, so that $z \in X_3$, as claimed.

\begin{prop} \label{Lemma outside the cones}
$X_3^{1,\Lambda} = X_3$. 
\end{prop}

\begin{proof}
Suppose $z_1, z_2 \in X_3$ satisfy $\0 \neq z_1-z_2 \in \Lambda$. We already mention that by Propositions \ref{Lemma 2 outside the cones}--\ref{Lemma 4 outside the cones} below, for every $(\rho,\v,\m) \in [z_1,z_2]$  we have $-1 < \rho - (1-\rho^2) v_2/\abs{\v}^2 < 1$.

Let $0 < \lambda < 1$ and $\lambda + \mu = 1$. We write
\begin{align*}
z
&=\lambda z_1 + \mu z_2 \\
&= \lambda \left( \rho + \mu t, \v + \mu \w, k_1 (\v + \mu \w) \right) + \mu \left( \rho - \lambda t, \v - \lambda \w, k_2 (\v - \lambda \w) \right) \\
&= \left( \rho, \v, k \v \right)
\end{align*}
and wish to show that $k = \rho - (1-\rho^2) v_2/\abs{\v}^2$.
We write
\begin{equation} \label{Difference of points in Y}
z_1 - z_2 = \left( t, \w, (k_1-k_2) \v + (\mu k_1 + \lambda k_2) \w \right) \in \Lambda.
\end{equation}
Corollary \ref{Wave cone corollary} and the assumption $z_1-z_2 \neq \0$ imply that $t \neq 0$. Assume, without loss of generality, that $t > 0$.

We first note that if $\w = \0$, then Corollary \ref{Wave cone corollary} yields $(k_1-k_2) v_2 = 0$. First, in the case $v_2 = 0$, then the assumption $z_1,z_2 \in X_3$ yields $k_1 = \rho + \mu t$ and $k_2 = \rho - \lambda t$, so that $k = \rho$ and $z \in X_3$.

We then treat the rest of the cases. Suppose, therefore, that either $\w \neq \0$ or $k_1-k_2 = \abs{\w} = 0$. In each case, by \eqref{Difference of points in Y} and Corollary \ref{Wave cone corollary}, we may write
\[z_1-z_2 = \left( t, \frac{t}{2} (\e-(0,1)), \ell (\e-(0,1)) \right)\]
for some $\e \in S^1$ and $\ell \in \R$.

We intend show that $k_1 < k_2$. (In particular, this rules out the case $k_1-k_2=\abs{\w}=0$.) This reduces to showing a claim that we next specify. Suppose
\[\zeta_1 = \left( \rho, \v, \left[ \rho - \frac{(1-\rho^2) v_2}{\abs{\v}^2} \right] \v \right) \eqdef (\rho, \v, \ell_1 \v) \in X_3,\]
that is,
\begin{equation} \label{Outside cones}
\abs{\v}^2 + \rho v_2 - \abs{v_2} > 0.
\end{equation}
Suppose $\epsilon > 0$ is small and $\zeta_2 = (\rho + \epsilon,\v + \epsilon (\e-(0,1))/2,\ell_2 [\v + \epsilon (\e-(0,1))/2]) \in X_3$, that is,
\begin{align*}
\zeta_2
&= \left( \rho + \epsilon, \v + \frac{\epsilon}{2} [\e-(0,1)], \right. \\
& \left. \left[ \rho+\epsilon - \frac{[1-(\rho+\epsilon)^2] \left( v_2 + \frac{\epsilon}{2} (e_2-1) \right)}{\abs{\v + \frac{\epsilon}{2} [\e-(0,1)]}^2} \right] \left( \v + \frac{\epsilon}{2} [\e-(0,1)] \right) \right),
\end{align*}
where $\abs{\e} = 1$. We claim that $\ell_1 < \ell_2$.

We write $\ell_2 - \ell_1$ as a Taylor series:
\begin{align*}
\ell_2-\ell_1
&= \rho+\epsilon - \frac{(1-\rho^2) v_2 + \epsilon [(1-\rho^2) (e_2-1)/2 - 2 \rho v_2] + O(\epsilon^2)}{\abs{\v}^2 + \epsilon \v \cdot [\e-(0,1)] + O(\epsilon^2)} \\
&- \rho + \frac{(1-\rho^2) v_2}{\abs{\v}^2} \\
&= \epsilon \left[ 1 - \frac{(1-\rho^2) (e_2-1)/2 - 2 \rho v_2}{\abs{\v}^2} + \frac{(1-\rho^2) v_2 \v \cdot [\e-(0,1)]}{\abs{\v}^4} \right] + O(\epsilon^2) \\
&= \frac{\epsilon}{\abs{\v}^2} \left[ \abs{\v}^2 + 2 \rho v_2 + \frac{(1-\rho^2) v_1 v_2 e_1}{\abs{\v}^2} + \frac{(1-\rho^2) (v_2^2-v_1^2) (e_2-1)/2}{\abs{\v}^2} \right] \\
&+ O(\epsilon^2) \\
&= \frac{\epsilon}{\abs{\v}^2} \left[ \abs{\v}^2 + 2 \rho v_2 + \frac{1-\rho^2}{2} \left( \frac{2 v_1 v_2}{\abs{\v}^2}, \frac{v_2^2-v_1^2}{\abs{\v}^2} \right) \cdot [\e-(0,1)] \right] + O(\epsilon^2).
\end{align*}
Thus it suffices to show that
\[H(\tilde{\e}) \defeq \abs{\v}^2 + 2 \rho v_2 + \frac{1-\rho^2}{2} \left( \frac{2 v_1 v_2}{\abs{\v}^2}, \frac{v_2^2-v_1^2}{\abs{\v}^2} \right) \cdot [\tilde{\e}-(0,1)] > 0 \quad \text{for all } \tilde{\e} \in S^1.\]
Note that $H$ is minimised when $\tilde{\e} \cdot (2 v_1 v_2/\abs{\v}^2,[v_2^2-v_1^2]/\abs{\v}^2)$ is minimised, that is, when $\tilde{\e} = - (2v_1v_2/\abs{\v}^2,[v_2^2-v_1^2]/\abs{\v}^2)$. The minimum value
\begin{align*}
H \left( \frac{-2v_1v_2}{\abs{\v}^2},\frac{v_1^2-v_2^2}{\abs{\v}^2} \right)
&= \abs{\v}^2 + 2 \rho v_2 - \frac{1-\rho^2}{2} - \frac{1-\rho^2}{2} \frac{v_2^2-v_1^2}{\abs{\v}^2} \\
&= \abs{\v}^2 + 2 \rho v_2 - (1-\rho^2) \frac{v_2^2}{\abs{\v}^2} \\
&= \frac{[\abs{\v}^2 + (\rho-1) v_2] [\abs{\v}^2 + (\rho+1) v_2]}{\abs{\v}^2} \\
&> 0
\end{align*}
by \eqref{Outside cones}. Thus $\ell_1 < \ell_2$. We conclude that $k_1 < k_2$ in \eqref{Difference of points in Y}, and so $\w \neq \0$.

\vspace{0.2cm}

Recall that $\w = t (\e-(0,1))/2$ for some $\e \in S^1$; since $\w \neq \0$, we have $\e \neq (0,1)$. Since we already showed that $k_1 < k_2$, we conclude from \eqref{Difference of points in Y} that $\v \cdot \w^\perp = 0$.

Proposition \ref{Lemma 4 outside the cones} below implies that $\v \neq \0$. Thus $\w = \ell \v$, where $\abs{\w + (0,t/2)} = \abs{t/2}$ gives $\ell = -t v_2/\abs{\v}^2$, so that we can write
\begin{align*}
z_1 &= \left( \rho + \mu t, \left( 1 - \mu \frac{t v_2}{\abs{\v}^2} \right) \v, k_1 \left( 1 - \mu \frac{t v_2}{\abs{\v}^2} \right) \v \right), \\
z_2 &= \left( \rho - \lambda t, \left( 1 + \lambda \frac{t v_2}{\abs{\v}^2} \right) \v, k_2 \left( 1 + \lambda \frac{t v_2}{\abs{\v}^2} \right) \v \right).
\end{align*}
Since $z_1, z_2 \in X_3$, we have
\begin{align*}
k_1
&= \rho + \mu t - \frac{[1-(\rho + \mu t)^2] v_2}{\left( 1 - \mu \frac{t v_2}{\abs{\v}^2} \right) \abs{\v}^2}, \\
k_2
&= \rho - \lambda t - \frac{[1-(\rho - \lambda t)^2] v_2}{\left( 1 + \lambda \frac{t v_2}{\abs{\v}^2} \right) \abs{\v}^2}
\end{align*}
so that
\begin{align*}
k
&= \lambda \left( 1 - \mu \frac{t v_2}{\abs{\v}^2} \right) k_1 + \mu \left( 1 + \lambda \frac{t v_2}{\abs{\v}^2} \right) k_2 \\
&= \rho - \lambda \mu \frac{t^2 v_2}{\abs{\v}^2} - \frac{v_2}{\abs{\v}^2} [\lambda [1-(\rho + \mu t)^2] + \mu [1- (\rho - \lambda t)^2]] \\
&= \rho - \frac{(1-\rho^2) v_2}{\abs{\v}^2},
\end{align*}
as claimed.
\end{proof}

\begin{prop} \label{Lemma 2 outside the cones}
$[X_3 - (X_2 \cup X_4)] \cap \Lambda = \emptyset$.
\end{prop}

\begin{proof}
Suppose
\begin{align*}
z_1 &= \left( \rho, \v, \left[ \rho - \frac{(1-\rho^2) v_2}{\abs{\v}^2} \right] \v \right) \eqdef (\rho,\v,k\v) \in X_3, \\
z_2 &= \left( \psi, \w, \ell \w \right) \in X_2 \cup X_4,
\end{align*}
so that $\v,\w \neq \0$. Seeking a contradiction, assume that
\[z_1-z_2 = (\rho-\psi, \v-\w, k(\v-\w) + (k-\ell) \w) \in \Lambda.\]
By the definitions of $X_2$ and $X_4$, we get $k \neq \ell$, so that Proposition \ref{Wave cone proposition} gives $\w \cdot \v^\perp = 0$. Now $\v = (1+t) \w$ for some $t \in \R \setminus \{-1,0\}$; if we had $t = 0$, then $z_1-z_2 \in \Lambda$ would imply $\rho = \psi$, in contradiction with the definitions of $X_2$, $X_3$ and $X_4$.

Now, since $z_1-z_2 \in \Lambda$, we have
\[0 = \abs{\v-\w}^2 + (\rho-\psi) (v_2-w_2) = t^2 \abs{\w}^2 + t (\rho-\psi) w_2\]
so that $\v = (1+t) \w = [1 + (\psi-\rho) w_2/\abs{\w}^2] \w$ and $\rho \neq \psi$. We therefore obtain
\begin{equation} \label{Formula for k}
k = \rho - \frac{(1-\rho^2) v_2}{\abs{\v}^2}
= \rho - \frac{(1-\rho^2) w_2}{\abs{\w}^2 + (\psi-\rho) w_2}.
\end{equation}
We divide the rest of the proof into separate cases.

Suppose first $z_2 \in X_2$ (that is, $\abs{\w}^2 + (1+\psi) w_2 \le 0$) and $1 + t > 0$ (i.e. $\abs{\w}^2 + (\psi-\rho) w_2 > 0$). By \eqref{Formula for k}, the assumption $k < 1$ can be written as $\abs{\w}^2 + (1+\psi) w_2 > 0$, which gives a contradiction.

Suppose next $\abs{\w}^2 + (1+\psi) w_2 \le 0$ and $\abs{\w}^2 + (\psi-\rho) w_2 < 0$. Thus $w_2 < 0$. Now $k > -1$ can be written as $\abs{\w}^2 + (\psi-1) w_2 < 0$, yielding a contradiction.

Similarly, if $z_2 \in X_4$ (i.e. $\abs{\w}^2 + (1+\psi) w_2 \le 0$) and $1 + t > 0$, then $k < 1$ is in contradiction with the assumption $z_2 \in X_4$. Finally, if $z_2 \in X_4$ and $1 + t < 0$, then $k > -1$ contradicts 
$z_2 \in X_4$.
\end{proof}

\begin{prop} \label{Lemma 4 outside the cones}
Suppose $z_1 \in X_1$, $z_2 \in X_2 \cup X_3 \cup X_4$ and, $z_2 - z_1 \in \Lambda$. Then the half-open interval $(z_1,z_2] \subset X_2 \cup X_4$.
\end{prop}

\begin{proof}
Suppose $z_1 = (\rho,\0,\m) \in X_1$ and $z_2 \in X_2 \cup X_3 \cup X_4$ satisfy $z_2-z_1 \in \Lambda$. Let $z = (\rho+\epsilon,\v,\tilde{\m}) \in (z_1,z_2]$; thus
\[(z-z_1) = (\epsilon,\v,\tilde{\m} - \m) \in \Lambda.\]
Also note that $z_2 \in X_2 \cup X_3 \cup X_4$ implies that $\v \neq \0$.

If $\epsilon = 0$, we get $z-z_1 = (0,\v,\tilde{\m}-\m) \in \Lambda$, which contradicts Corollary \ref{Wave cone corollary}. We then assume that $0 < \epsilon \le 1-\rho$. By Proposition \ref{Wave cone proposition}, $\abs{\v}^2 + \epsilon v_2 = 0$. Since $\v \neq \0$, we conclude that $v_2 < 0$. Thus
\[\abs{\v}^2 + (\rho + \epsilon + 1) v_2 = (\rho+1) v_2 \le 0\]
which, combined with Corollary \ref{Corollary on Lambda-convex function for first cone}, yields $z \in X_2$. Similarly, if $-1-\rho \le \epsilon < 0$, then $z_2 \in X_4$.
\end{proof}

We finish the proof of Theorem \ref{Main theorem} by showing that a $\Lambda$-segment between $z_1 \in X_2$ and $z_2 \in X_4$ cannot contain $(\rho,\v,k\v)$ with $\v \neq \0$ and $-1 < \rho - (1-\rho^2) v_2/\abs{\v}^2 < 1$.

\begin{prop} \label{Lemma 3 outside the cones}
$(X_2 \cup X_4)^{1,\Lambda} \subset X_1 \cup X_2 \cup X_4$.  
\end{prop}

\begin{proof}
Suppose
\[z_1 = (\rho,\v,k\v) \in X_2, \qquad
z_2 = \left( \psi, \w, \ell \w \right) \in X_4\]
and
\[z_1-z_2 = (\rho-\psi, \v-\w, k(\v-\w) + (k-\ell) \w) \in \Lambda.\]
Thus
\[1 \le k \le \rho - \frac{(1-\rho^2) v_2}{\abs{\v}^2}, \qquad \psi - \frac{(1-\psi^2) w_2}{\abs{\w}^2} \le \ell \le -1,\]
giving $\abs{\v}^2 + (\rho+1) v_2 \le 0$ and $\abs{\w}^2 + (\psi-1) w_2 \le 0$, which in turn yields $v_2 < 0 < w_2$. Now $\rho \neq \psi$, as otherwise $z_1-z_2 \in \Lambda$ would give $\v-\w = \0$, contradicting $v_2 < 0 < w_2$.

Choose the unique $\tilde{\psi} = \lambda \psi + \mu \rho \in [\psi,\rho]$ (where $0 \le \lambda \le 1$ and $\lambda + \mu = 1$) such that $\tilde{z} = (\tilde{\psi}, \tilde{\w}, \tilde{\m}) \defeq \lambda z_1 + \mu z_2$ satisfies $\tilde{\w} = \0$ or
\begin{equation} \label{Boundary of cone}
\tilde{\psi} - \frac{(1-\tilde{\psi}^2) \tilde{w}_2}{|\tilde{\w}|^2} = -1.
\end{equation}
If $\tilde{\w} = \0$, then $\tilde{z} \in X_1$ and we are reduced to the situation of Proposition \ref{Lemma 4 outside the cones}. Assume, therefore, $\tilde{\w} \neq \0$ and \eqref{Boundary of cone} holds. Consequently, $|\tilde{\w}|^2 + (\tilde{\psi}-1) \tilde{w}_2 = 0$, giving $\tilde{w}_2 > 0$. Note that \eqref{Boundary of cone} and Corollary \ref{Corollary on Lambda-convex function for first cone} give $\tilde{z} = (\tilde{\psi}, \tilde{\w}, -\tilde{\w})$.

Now, by assumption,
\[z_1 - \tilde{z} = (\rho-\tilde{\psi},\v-\tilde{\w}, k(\v-\tilde{\w}) + (k+1) \tilde{\w}) \in \Lambda,\]
so that $\tilde{\w} \cdot \v^\perp = 0$ since $k \ge 1$. Let us write $\v = (1+t) \tilde{\w}$; now $z_1-z_2 \in \Lambda$ gives $t = (\tilde{\psi}-\rho) \tilde{\w}_2/|\tilde{\w}|^2$. On the other hand, $v_2 < 0 < \tilde{w}_2$ and $\v = (1+t) \tilde{\w}$ yield $1+t < 0$, so that
\[0 > \frac{|\tilde{\w}|^2 + (\tilde{\psi}-\rho) \tilde{\psi}_2}{|\tilde{\w}|^2} = \frac{(1-\rho) \tilde{w}_2}{|\tilde{\w}|^2},\]
giving a contradiction with $\tilde{w}_2 > 0$.
\end{proof}

This finishes the proof of Theorem \ref{Main theorem} and gives the exact description of the lamination convex hull of the stationary IPM equations. Furthermore, outside the 'rigid region' of $K^{lc,\Lambda}$ where $(\rho,\v) \in \R \times (\R^2 \setminus \{\0\})$ with $\abs{\v}^2 + \rho v_2 > \abs{v_2}$ we get the same description for the $\Lambda$-convex hull. If we could get this result for all $(\rho,\v) \in \R \times \R^2$, we could formulate Theorem \ref{Second main theorem} for the $\Lambda$-convex hull instead of the lamination convex hull.

\section{Non-existence of non-trivial subsolutions in bounded domains} \label{Non-existence of non-trivial solutions in bounded domains}
As observed in ~\cite{Elgindi} (although stated under different hypotheses), if $\v \in L^2_\sigma(\Omega,\R^2)$ and $\rho \in L^\infty(\Omega)$ form a solution of stationary IPM, then
\begin{equation} \label{Elgindi's computation}
\int_\Omega \abs{\v}^2 = \int_\Omega \v \cdot [-\nabla p - (0,\rho)] = - \int_\Omega \rho v_2 = - \int_\Omega \rho \v \cdot \nabla y = 0.
\end{equation}
We adapt the proof to subsolutions with values in $K^{lc,\Lambda}$ by using the exact form of $K^{lc,\Lambda}$ computed in Theorem \ref{Main theorem}.

\begin{proof}[Proof of Theorem \ref{Second main theorem}]
Since $\m \in L^2_\sigma(\Omega,\R^2) = [\nabla W^{1,2}(\Omega)]^\perp$ and $(\rho,\v,\m)(x) \in K^{lc,\Lambda}$ a.e. $x \in \Omega$, we may write
\begin{equation} \label{Vanishing integral with k}
0 = \int_{\Omega} \m \cdot \nabla y = \int_{\v = \0} \frac{1-\rho^2}{2} (e_2-1) + \sum_{j=2}^4 \int_{(\rho,\v) \in X_j} k v_2.
\end{equation}
If $(\rho,\v) \in X_2$, then $1 \le k \le \rho - (1-\rho^2) v_2/\abs{\v}^2$ so that either $\rho = 1$ or $v_2 < 0$. In both cases, $k v_2 \le v_2$. Thus
\begin{equation} \label{Inequality for first cone}
\int_{(\rho,\v) \in X_2} k v_2 \le \int_{(\rho,\v) \in X_2} v_2.
\end{equation}
Similarly, if $(\rho,\v) \in X_4$, then $\rho - (1-\rho^2) v_2/\abs{\v}^2 \le k \le -1$ so that either $\rho = -1$ or $v_2 > 0$, giving $k v_2 \le -v_2$ and
\begin{equation} \label{Inequality for second cone}
\int_{(\rho,\v) \in X_4} k v_2 \le -\int_{(\rho,\v) \in X_4} v_2.
\end{equation}
Furthermore, since $(\rho,\v,\m)(x) \in K^{lc,\Lambda}$ a.e. $x \in \Omega$, we get
\[\int_{(\rho,\v) \in X_3} k v_2 = \int_{(\rho,\v) \in X_3} \rho v_2 - \int_{(\rho,\v) \in X_3} \frac{(1-\rho^2) v_2^2}{\abs{\v}^2}.\]
Using \eqref{Vanishing integral with k}--\eqref{Inequality for second cone},
\begin{align*}
-\int_{(\rho,\v) \in X_3} \rho v_2
&= - \int_{(\rho,\v) \in X_3} k v_2 - \int_{(\rho,\v) \in X_3} \frac{(1-\rho^2) v_2^2}{\abs{v}^2} \\
&= \int_{\v = \0} \frac{1-\rho^2}{2} (e_2-1) + \int_{(\rho,\v) \in X_2 \cup X_4} k v_2 - \int_{X_3} \frac{(1-\rho^2) v_2^2}{\abs{v}^2} \\
&\le \int_{(\rho,\v) \in X_2} v_2 - \int_{(\rho,\v) \in X_4} v_2 - \int_{(\rho,\v) \in X_3} \frac{(1-\rho^2) v_2^2}{\abs{\v}^2},
\end{align*}
and so, using the assumption that $\v \in L^2_\sigma(\Omega,\R^2)$,
\begin{align*}
0
&\le \int_{\Omega} \abs{\v}^2 = \int_\Omega \v \cdot [-\nabla p - (0,\rho)] = - \int_{\Omega} \rho v_2 = - \sum_{j=2}^4 \int_{(\rho,\v) \in X_j} \rho v_2 \\
&\le \int_{(\rho,\v) \in X_2} v_2 - \int_{(\rho,\v) \in X_4} v_2 - \int_{(\rho,\v) \in X_3} \frac{(1-\rho^2) v_2^2}{\abs{\v}^2} \\
&- \int_{(\rho,\v) \in X_2} \rho v_2 - \int_{(\rho,\v) \in X_4} \rho v_2 \\
&= \int_{(\rho,\v) \in X_2} (1-\rho) v_2 - \int_{(\rho,\v) \in X_4} (1+\rho) v_2 - \int_{(\rho,\v) \in X_3} \frac{(1-\rho^2) v_2^2}{\abs{\v}^2} \\
&\le 0,
\end{align*}
where in the last inequality we have used $v_2\leq 0$ in $X_2$ and $v_2 \geq 0$ in $X_4$. We thus conclude that $\v = \0$. Now \eqref{IPM 3} gives $\partial_x \rho = \nabla^\perp \cdot (0,\rho) = 0$.
\end{proof}

\begin{rem} \label{Dichotomy remark}
The proof of Theorem \ref{Second main theorem} also works essentially verbatim with impermeable walls in the vertical direction and periodic boundary conditions in the horizontal direction. Thus, the dichotomy on directions of strips that we mentioned in the introduction extends to subsolutions with values in $K^{lc,\Lambda}$.

Adapting \eqref{Elgindi's computation} to a strip with finite width in the direction $(0,1)$, we briefly indicate the role that the direction $(0,1)$ plays. The second equality in \eqref{Elgindi's computation} uses the boundary conditions that $\v \cdot \nu|_{\partial \Omega} = 0$ when $y = 0$ and $\v$ is periodic in $x$; this part works equally in the setting of ~\cite{CLV}. However, the fourth equality in \eqref{Elgindi's computation} uses the fact that $(x,y) \mapsto y$ is periodic in $x$. It is here that the adaptation to all other strips breaks down, and thus there is no geometric obstruction to the solutions of ~\cite{CLV}. In the proof of Theorem \ref{Second main theorem}, the fourth equality of \eqref{Elgindi's computation} is necessarily replaced by a weaker condition, and the proof requires the precise computation of $K^{lc,\Lambda}$ in Theorem \ref{Main theorem}.
\end{rem}

\section{Relation to the infinite time limit of non-stationary IPM} \label{Relation to the limit of non-stationary IPM}
As the last topic of this paper, we show that Theorem \ref{Second main theorem} reflects the behaviour of subsolutions of non-stationary IPM at the limit $t \to \infty$. The proof is a straightforward application of ~\cite[Corollary 1.2]{Elgindi} which states that $\partial_t \int_\Omega \rho x_2 \d x = 2^{-1} \partial_t \int_\Omega \abs{\rho-x_2}^2 \d x = - \int_\Omega |\v|^2 \d x$ for smooth solutions of non-stationary IPM.

\begin{prop} \label{Third main theorem}
Suppose $\rho \in L^\infty(0,\infty;L^\infty)$ and $\v,\m \in L^\infty(0,\infty;L^2_\sigma)$ form a subsolution of non-stationary IPM in a smooth, bounded, simply connected domain $\Omega \subset \R^2$. Then $\v \in L^2(0,\infty;L^2_\sigma)$.
\end{prop}

Proposition \ref{Third main theorem} and its proof work equally well in the confined IPM case $\Omega = \T^1 \times (-1,1)$. Before presenting the proof, we recall the definition of a subsolution in this context. Under the integrability assumptions of Theorem \ref{Third main theorem}, $z = (\rho,\v,\m)$ is a subsolution of non-stationary IPM if
\begin{equation} \label{z in hull}
z(x) \in K^\Lambda = \left\{ (\bar{\rho},\bar{\v},\bar{\m}) \colon |\bar{\rho}| \le 1, \abs{\bar{\m} - \bar{\rho} \bar{\v} + \left(0,\frac{1-\bar{\rho}^2}{2} \right)} \le \frac{1-\bar{\rho}^2}{2} \right\}
\end{equation}
a.e. $x \in \Omega \times [0,\infty)$ and
\begin{align}
& \int_0^\infty \int_\Omega (\rho \, \partial_t \varphi + \m \cdot \nabla \varphi) \d x \d t + \int_\Omega \rho_0 \varphi(\cdot,0) \d x = 0 \quad \forall \varphi \in C_c^\infty(\bar{\Omega} \times [0,\infty)), \label{Subsolution of non-stationary IPM 1} \\
& \int_0^\infty \int_\Omega \v \cdot \nabla \varphi \d x \d t = 0 \quad \forall \varphi \in C_c^\infty(\bar{\Omega} \times [0,\infty)), \label{Subsolution of non-stationary IPM 2} \\
& \int_0^\infty \int_\Omega (\v + (0,\rho)) \cdot \nabla^\perp \varphi \d x \d t = 0 \quad \forall \varphi \in C_c^\infty(\Omega \times [0,\infty)). \label{Subsolution of non-stationary IPM 3}
\end{align}
Note that \eqref{Subsolution of non-stationary IPM 1}--\eqref{Subsolution of non-stationary IPM 2} incororate the condition $\v \cdot \nu|_{\partial \Omega} = \m \cdot \nu|_{\partial \Omega} = 0$.

\begin{proof}[Proof of Proposition \ref{Third main theorem}]
Let $\eta \in C_c^\infty(0,\infty)$ and set $\varphi(x,t) \defeq \eta(t) x_2$ in \eqref{Subsolution of non-stationary IPM 1}, so that
\[\int_0^\infty \eta'(t) \int_\Omega \rho(x,t) x_2 \d x \d t + \int_0^\infty \eta \int_\Omega m_2(x,t) \d x \d t = 0.\]
As a consequence, $\partial_t \int_\Omega \rho(x,\cdot) x_2 \d x = \int_\Omega m_2(x,\cdot) \d x \in L^\infty(0,\infty)$ in the sense of distributions. Thus, after possibly modifying $\rho$ on a set of measure zero, $F(t) \defeq \int_\Omega \rho(x,t) x_2 \d x$ is Lipschitz continuous and
\begin{equation} \label{Information on F 1}
F(t) = \int_\Omega \rho_0(x) x_2 \d x + \int_0^t \int_\Omega m_2(x,\tau) \d x \d \tau
\end{equation}
for all $t \in [0,\infty)$.

We use \eqref{z in hull} to get $m_2 = \rho v_2 + (1-\rho^2) (e_2-1)/2$, where $\e = (e_1,e_2)$ takes values in $\bar{B}(0,1)$, so that
\begin{equation} \label{Information on F 2}
\int_0^t \int_\Omega m_2(x,\tau) \d x \d \tau
\le \int_0^t \int_\Omega \rho(x,\tau) v_2(x,\tau) \d x \d \tau.
\end{equation}
Now, approximating $\v$ in $L^2(0,t;L^2_\sigma)$ by mappings $\nabla^\perp \varphi_j$, $\varphi_j \in C_c^\infty(\Omega \times [0,t))$, the assumption \eqref{Subsolution of non-stationary IPM 3} gives
\begin{equation} \label{Information on F 3}
\int_0^t \int_\Omega \rho(x,\tau) v_2(x,\tau) \d x \d \tau = - \int_0^t \int_\Omega \abs{\v(x,\tau)}^2 \d x \d \tau.
\end{equation}
Combining \eqref{Information on F 1}--\eqref{Information on F 3}, we conclude that
\begin{align*}
\int_0^t \int_\Omega |\v(x,\tau)|^2 \d x \d \tau - \int_\Omega \rho_0(x) x_2 \d x
&\le -F(t) \le \int_\Omega |\rho(x,t) x_2| \d x \\
&\le \norm{\rho}_{L^\infty(0,\infty;L^\infty)} \int_\Omega |x_2| \d x
\end{align*}
for all $t \in [0,\infty)$. The claim follows.
\end{proof}

\bigskip
\footnotesize
\noindent\textit{Acknowledgments.}
We express warm thanks to \'{A}ngel Castro, Daniel Faraco and Francisco Mengual for useful comments.

\bibliography{Stationaryipm}
\bibliographystyle{amsplain}
\end{document}